\documentclass[12pt]{article}
\usepackage{amsmath,amsfonts,amssymb,amsthm,amscd}
\title{Thorn forking,  weak normality, and theories with selectors}
\date{\today}
\author
{Daniel Max Hoffmann\thanks{SDG. Supported by the Polish National Agency for Academic Exchange}\\University of Notre Dame\\ University of Warsaw    \and Anand Pillay\thanks{Supported by NSF grants  DMS-1665035, DMS-1760212 and DMS-2054271, as well as by a Nelder Visiting Fellowship at Imperial College, London, where final corrections were done. }\\University of Notre Dame }

\newtheorem{Theorem}{Theorem}[section]
\newtheorem{Proposition}[Theorem]{Proposition}

\newtheorem{Remark}[Theorem]{Remark}
\newtheorem{Lemma}[Theorem]{Lemma}
\newtheorem{Corollary}[Theorem]{Corollary}

\newtheorem{Question}[Theorem]{Question}

\DeclareMathOperator{\aut}{Aut}
\begin{document}
\maketitle

\begin{abstract}  We discuss the role of weakly normal formulas in the theory of thorn forking, as part of a commentary on the paper \cite{Ealy-Onshuus}.  We also give a  counterexample to Corollary 4.2 from that paper, and in the process discuss ``theories with selectors". 
\end{abstract}

\section{Introduction}

This work comes out of our reading of the paper \cite{Ealy-Onshuus}. A superficial look  at \cite{Ealy-Onshuus} may give the impression (as it gave to the authors) that the results in that paper shed some light on the stable forking conjecture. 
We point out that it is a stronger version of stability, {\em weak normality}, which is relevant, more or less by definition, although there is still no real connection with the stable forking conjecture. We also point out a counterexample to Corollary 4.2 of  \cite{Ealy-Onshuus} which claims that strong dividing of $\phi(x,c)$ can be witnessed by parameters from $dcl(c)$.

We will work as usual in a saturated model ${\bar M} = {\bar M}^{eq}$ of a complete first order theory $T$ in language $L$. (Namely we  assume that $T = T^{eq}$.)
We could mean by this a $\kappa$-saturated and strongly $\kappa$-homogeneous model (for some large $\kappa$). On the other hand a recent note \cite{Halevi-Kaplan} makes precise the justification for the widespread practice of assuming the existence of saturated
models in the strict sense of $\kappa$-saturated and of cardinality $\kappa$ (for arbitrarily large $\kappa$).  So we feel free to assume ${\bar M}$ to be saturated in the strict sense.  This is relevant to Section 4, where we prove completeness of a theory $T^{+}$. However in all cases, back and forth arguments inside saturated models in the weaker sense will be enough, as we will mention.

 Our basic model-theoretic notation  is as in \cite{Pillay-book}.   $x,y, z,\ldots$ denote (finite tuples of) variables, and $a,b,c,\ldots$ (finite tuples of) elements of the ambient model. Definability means with parameters. Remember that a definable set $X$  is said to be ``almost over $A$" iff $X$ has finitely many images under $Aut({\bar M}/A)$  iff the canonical parameter of $X$ is in $acl(A)$.  We typically identify a set $A$ of parameters with its definable closure.
Following a request from the referee, let us be more precise about canonical parameters of definable sets. Given a definable set $X$, let $\phi(x,y)$ be an $L$-formula and $c$ a parameter such that $\phi(x,c)$ defines $X$. Relative to this choice of $\phi(x,y)$ we obtain a cano\-nical parameter for $X$: let $E_{\phi}(y_{1}, y_{2})$ be the $\emptyset$-definable equivalence relation $\forall x (\phi(x,y_{1})\leftrightarrow \phi(x,y_{2}))$. Then the equivalence class $c/E_{\phi}$ of $c$ is a canonical parameter for $X$, and note that we can find a formula with parameter $c/E_{\phi}$ which defines $X$. Any other choice of $L$-formula $\phi$ gives another canonical parameter for $X$, but all canonical parameters are interdefinable over $\emptyset$. 

We will be assuming a basic knowledge of notions from stability theory, such as forking, dividing, and canonical bases of stationary types, although we will repeat some of the definitions below.  In addition to \cite{Pillay-book} one could refer to \cite{Casanovas}.

Thorn forking was introduced by Onshuus \cite{Onshuus} as an account of a  general theory of independence, which subsumes nonforking in stable theories, and also in simple theories assuming elimination of hyperimaginaries, but also includes ``dimension-independence" in $o$-minimal theories.  A very nice account and explanation of thorn forking appears in \cite{Adler}. The stable forking conjecture says that stable formulas are responsible for forking in simple theories.  Here by a stable formula we mean an $L$-formula $\phi(x,y)$ where the free variables of $\phi$ are partitioned into $x,y$ and $\phi(x,y)$ does not have the order property.  And to say that stable formulas are responsible for forking means that whenever $p(x)\in S(B)$ forks over $C\subseteq B$,  then there is a stable $L$-formula $\phi(x,y)$ and an instance $\phi(x,b)$ of $\phi(x,y)$ which is in $p(x)$ and such that $\phi(x,b)$ forks over $C$. 

The aim of this paper is to clarify the  content of \cite{Ealy-Onshuus}, so we will first  say what the paper appears to be about (even though there is a typographical error in the main Definition 1.2 of stable dividing with parameters).  We will for the sake of completeness recall the basic notions. A formula $\phi(x,b)$ is said to {\em divide} over $A$ if it is consistent, and there is $k$ and an infinite sequence $(b_{i}:i<\omega)$ of realizations of  $tp(b/A)$  such that $\{\phi(x,b_{i}):i<\omega\}$ is $k$-inconsistent (for every choice of $i_1<\ldots<i_k<\omega$, the set $\{\phi(x,b_{i_{1}}),\ldots,\phi(x,b_{i_{k}})\}$ is inconsistent),  equivalently there is an indiscernible sequence $(b_{i}:i< \omega)$ of realizations of $tp(b/A)$ such that $\{\phi(x,b_{i}): i<\omega\}$ is inconsistent.  $\phi(x,b)$ is said to {\em strongly divide} over $A$, if $\phi(x,b)$ is not almost over $A$ and for some $k$, the set of {\em all} $A$-conjugates of $\phi(x,b)$ is $k$-inconsistent. The formula $\phi(x,b)$ {\em thorn divides} over $A$, if for some tuple $d$ of parameters, $\phi(x,b)$ strongly divides over $(A,d)$.  A complete type $p(x)\in S(B)$ (thorn) divides over $A\subseteq B$ if $p$ contains a formula which (thorn) divides over $A$.  

Finally, one says that a complete type $p(x)\in S(B)$ (thorn) forks over $A$ if $p$ implies a finite disjunction of formulas each of which (thorn) divides over $A$. We freely use the obvious (by compactness) fact that $p(x)\in S(B)$ does not (thorn) fork over $A$ if and only if for any ${\hat B}\supseteq B$, $p(x)$ has an extension to a complete type $q(x)$ over ${\hat B}$ which does not (thorn) divide over $A$.

Theorem 4.1 of \cite{Ealy-Onshuus} says that thorn-forking is unchanged if one {\em redefines}  a formula $\phi(x,b)$ to thorn divide over $A$, if for some tuple $d\in dcl(A,b)$, $\phi(x,b)$ strongly divides over $(A,d)$. Corollay 4.2 of \cite{Ealy-Onshuus} is supposed to follow immediately from Theorem 4.1 and states that if $\phi(x,b)$ thorn divides over $A$ then the additional parameters $d$ witnessing strong dividing can be found in $dcl(A,b)$. We give a counterexample to this Corollary 4.2 which suggests to us that Theorem 4.1 is also false. 

Theorem 3.3 of \cite{Ealy-Onshuus} depends on a notion of stably dividing with parameters (Definition 1.2 in \cite{Ealy-Onshuus}).  After correcting a typographical error, we restate the definition as follows: Let $X$ be a definable set and $A$ a set of parameters. Then $X$ stably divides with parameters (w.p), over $A$, if there is a set $B$ of parameters, such that 
\newline 
(1) For some stable $L_{B}$-formula $\phi(x,y)$ and parameter $c$, $X$ is defined by $\phi(x,c)$, and 
\newline
(2) $X$ divides over $A\cup B$.

One recovers as before a notion ``$p(x)\in S(B)$ stably forks with parameters", over $A$. Theorem 3.3 states that thorn forking coincides with stable forking with parameters. 

As mentioned above, and we point out below in Proposition 3.1,  it is rather a {\em special case} of stable formulas, namely weakly normal formulas, which are relevant, and this is essentially immediate from the definition of thorn forking. 
In any case we also give an independent proof of Theorem 3.3 of \cite{Ealy-Onshuus}. 
 

Before continuing, let us mention some apparent deficiencies in the formulations in \cite{Onshuus}, especially concerning strong dividing, which reappear in \cite{Ealy-Onshuus}. 
The actual definition in \cite{Onshuus} depends on the particular formula $\phi(x,b)$ rather than the set it defines, and says that $\phi(x,b)$ strongly divides over $A$ if 
$b\notin acl(A)$ and 
the set of formulas $\phi(x,b')$, as $b'$ ranges over all distinct realizations of $tp(b/A)$ is $k$-inconsistent, for some $k$.
Let us also recall that $\phi(x,b)$ thorn divides over $A$ if $\phi(x,b)$ strongly divides over $Ad$ for some finite tuple $d$.
So adding dummy parameters to $b$ can preserve the definable set, but change strong dividing (for different reasons).  On the other hand, if we define thorn forking  with this definition of a formula thorn dividing, we DO get the same notion as when we want $b$ to be a canonical parameter of $\phi(x,b)$, so nothing is really lost. 

In any case, in the current paper we use the correct functorial definitions, depending only on the definable set, not the choice of a formula defining it. 

We would like to thank the referee for a very helpful report, pointing out  errors and issues requiring clarification.

\section{Weak normality}
As in Chapter 4 of \cite{Pillay-book} we call a definable (with parameters) set $X$ {\em weakly normal} if $X$ is nonempty, and  for any $a\in X$ there are only finitely many $Aut({\bar M})$ conjugates of $X$ which contain $a$.   When $X$ is already defined over $acl(\emptyset)$ this condition trivially holds.
Consider a canonical parameter $e$ for the set $X_{e}:=X$.
Weak normality is equivalent to saying that for every $e'$ sharing with $e$ its  type over the empty set, and for every $a\in X_{e'}$ (the shift of $X_e$) we have that
$e'\in acl(a)$.
The condition that $X$ is not over $acl(\emptyset)$ means that $e\notin acl(\emptyset)$.

We can relativise weak normality to any (small) set $A$ of parameters:  $X$ is {\em weakly normal relative to $A$} if $X$ is weakly normal in the expansion of ${\bar M}$ by adding constants for elements of $A$.
This corresponds to saying that for every $e'\equiv_A e$ and every $a\in X_{e'}$
it must be that $e'\in acl(A,a)$.
The ``nontriviality condition", over $A$, means that $e\notin acl(A)$. By compactness we obtain the following characterization of weak normality.

\begin{Remark}\label{remark:k.many.weakly.normal}
Assume that $X$ is a definable set.
The definable set  $X$ is weakly normal relative to a set $C$ if and only if
there exists a natural number $k>0$ such that 
for every $k$-many distinct $C$-conjugates of $X$, the intersection of these $C$-conjugates is empty (in particular, if the number of distinct $C$-conjugates of $X$ is finite).
\end{Remark}

Hence the definable set $X$ strongly divides over $C$ iff $X$ is weakly normal relative to $C$ and is not almost over $C$.

\begin{Remark}\label{remark:old.2.1}
Let $X$ be a definable set which is weakly normal relative to a set $C$. Then $X$ divides over $C$ iff 
$X$ is not almost over $C$. 
\end{Remark}
\begin{proof} Clearly dividing over $C$ implies  not being almost over $C$  (for any definable set $X$, weakly normal relative to   $C$ or not).  Conversely suppose that $X$ is weakly normal relative to $X$ and that $X$ is not almost over $C$. 
Then by Remark \ref{remark:k.many.weakly.normal}, the set $X$ $k$-divides over $C$ for some $k$.
\end{proof}

\begin{Remark}\label{remark:old.2.2}
 Suppose we are given a tuple $a$ and sets $C\subseteq B$. Then the following are equivalent:
\newline
(i) There is $b$ in $dcl(B)$ which is in $acl(a,C)\setminus acl(C)$.
\newline
(ii) For some weakly normal relative to $C$, definable set $X$ in $tp(a/B)$, $X$ is not almost over $C$.
\end{Remark}

\begin{proof} Let $b\in dcl(B)$ be as in (i).  
We can then  find $L$-formulas $\phi(x,y,z)$ and $\psi(y,w)$, $c$ in $C$
and $b_0\in B$ such that  $\models\phi(a,b,c)\,\wedge\,\psi(b,b_0)$,
\begin{itemize}
\item 
$\phi(x,y,z)$ implies $(\exists^{=n}y')\,\phi(x,y',z)$ for some $n>0$

\item 
and $\psi(y,w)$ implies $(\exists^{=1}y')\,\psi(y,w)$.
\end{itemize}
Let $X$ be the set defined by $(\exists y)\,(\phi(x,y,c)\,\wedge\,\psi(y,b_0))$.
So $X$ is in $tp(a/B)$. Note that $X=\phi(\bar{M},b,c)$.
\newline
It follows from the properties of $\phi(x,y,z)$ and $\psi(y,w)$ that $X$ is weakly normal relative to $C$. We have to check that $X$ is not almost over $C$.  
Suppose for a contradiction that $X$ is almost over $C$.
As $b\not\in acl(C)$, choose an infinite sequence $f_i\in\aut(\bar{M}/C)$, $i<\omega$, such that all $f_i(b)$'s are distinct. 
Without loss of generality, we have $f_i(X)=f_1(X)\ni f_1(a)$ for infinitely many $i<\omega$ and
so $\models\phi(f_1(a),f_i(b),c)$ for infinitely many $i<\omega$ which contradicts the choice of $\phi$.
The implication from (ii) to (i) follows by definition, when we choose $b$ to be a canonical parameter of $X$.
\end{proof}

\begin{Remark} Let $X$ be a definable set, defined by formula $\phi(x,e)$ where $\phi(x,y)$ is an $L$-formula and $e$ is the canonical parameter of $X$. Let $C$ be a set of parameters.
The following are equivalent:
\newline
(i) $X$ is weakly normal relative to $C$,
\newline
(ii) There is a formula $\psi(y)\in tp(e/C)$ (so with parameters in $C$) such that the formula $\theta(x,y)$:= $\phi(x,y)\wedge \psi(y)$ has the property that for any $a$ there are only finitely many (up to equivalence) instances $\theta(x,e')$ of $\theta(x,y)$ such that $\models \theta(a,e')$ In particular the $L_{C}$-formula $\theta(x,y)$ is stable. 
\end{Remark}
\begin{proof} (ii) implies (i) is immediate, as $C$-conjugates of $X$ are instances of $\theta(x,y)$.
\newline
(i) implies (ii) is a compactness argument.

\end{proof}

Let's note that formulas $\theta(x,y)$ in (ii) above are not only stable but are also {\em equations} in the sense of Srour \cite{Srour}. 


\section{Thorn forking}

Here is our characterization of thorn forking in terms of weakly normal sets relative to parameters.  It is basically tautological.

\begin{Proposition} \label{proposition:thorn.weakly.normal}
Given tuple $a$ and sets $C\subseteq B$, $tp(a/B)$ does not thorn fork over $C$ iff there is some extension of $tp(a/B)$ to a complete type $p(x)$  over some $(|T|+|C|)^{+}$-saturated model $M$, such that whenever $X$ is a definable set in $p(x)$ and $C\subseteq C'\subseteq M$ and $X$ is weakly normal relative to $C'$ then $X$ is almost over $C'$ (equivalently $X$ does not divide over $C'$).
\end{Proposition}
\begin{proof}
By definition $tp(a/B)$ does  not thorn fork over $C$ if and only if for all ${\hat B}\supseteq B$, $tp(a/B)$ has an extension to a complete type $p$ over ${\hat B}$ which does not thorn divide over $C$. 
Let $M$ be a $(|T| +|C|)^{+}$-saturated model containing $B$.
\newline
{\em Claim 1.} The type $tp(a/B)$ does not thorn fork over $C$ iff $tp(a/B)$ has an extension  to a complete type $p(x)$ over $M$ which does not thorn divide over $C$.
\newline
{\em Proof of Claim 1.}  Clearly left implies right. Now, suppose the right hand side holds. Let $p_{0}(x) =tp(a/B)$.
Consider ${\hat B}\supseteq B$ and set $\Gamma(x) = p_{0}(x) \cup \{\neg\theta(x,b): \theta(x,y)$ an $L$-formula, $b\in {\hat B}$ and $\phi(x,b)$ thorn divides over $C$\}. We will show that $\Gamma$ is consistent. Let $\Gamma_{0}(x)$  be a finite subset of $\Gamma$ and let $B_{0}$ be the finite set of parameters from ${\hat B}$ which appear in $\Gamma_0$. We can realize $tp(B_{0}/C)$ in $M$ by $B_{0}'$  in such a way that the elements of $B_{0}$ which are in $B$ are realized by themselves.  Let $\Gamma_{0}'$ be the copy of $\Gamma_{0}$ over $B_{0}'$. Then $\Gamma_{0}'\subseteq p(x)$. So $\Gamma_{0}'$ is consistent, hence also $\Gamma_{0}$ is consistent. Thus $\Gamma$ is consistent as well and extends to a complete type over ${\hat B}$ extending $p_{0}$ and which does not thorn divide over $C$. 

\vspace{2mm}
\noindent
{\em Claim 2.} Let $p(x)\in S(M)$. Then $p$ does not thorn divide over $C$ iff for every $\theta(x)\in p(x)$, $\theta$ does not strongly divide over any set $(C,d)$ where $d$ is a tuple from $M$. 
\newline
{\em Proof of Claim 2.}  By definition, $p$ does not thorn divide over $C$ iff for every $\theta(x)\in p$ there is no tuple $d$ such that $\theta(x)$ strongly divides over $C,d$.  Now if $\theta(x)\in p$ and $\theta$ strongly divides over $C$ together with some tuple $d$, then by saturation of $M$ let $d'\in M$ realize the type of $d$ over $C$ together with the parameters from $\theta$, so $\theta$ strongly divides over $(C,d')$. This suffices. 

To finish the proof of the Proposition from Claim 2, we just have to note that if $X$ is a definable set, and $C'$ a set of parameters, then $X$ strongly divides over $C'$ if and only if $X$ is weakly normal relative to $C'$ and is not almost over $C'$, which  is stated after Remark 2.1.

Let us finally remark that the proof above of \ref{proposition:thorn.weakly.normal} is roughly the content of Proposition 4.7 of \cite{Adler}. 
\end{proof}

\begin{Remark}
A few years ago, Ludomir Newelski asked about absoluteness of the notion of the thorn forking,
i.e. whether this notion depends on the choice of the model of set theory. Using the above Proposition \ref{proposition:thorn.weakly.normal} (or alternatively some parts of \cite{Adler}),
one can show that indeed the notion of the thorn forking does not depend on the choice of the monster model. This was discussed with Newelski during a talk on this paper.
\end{Remark}

\vspace{4mm}
\noindent
Following the referee's request, we will give a  self-contained proof of Theorem 3.3 of \cite{Ealy-Onshuus} (which we understand to be the main result of \cite{Ealy-Onshuus}).  Recall from \cite{Ealy-Onshuus} (and our introduction) that for $A\subseteq B$ and $p(x)\in S(B)$, $p$ is said to ``stably fork with parameters" over $A$ if $p$ implies a finite disjunction of formulas $\psi_{i}(x)$ (over additional parameters) each of which ``stably divides with parameters" over $A$. Where again $\psi_{i}(x)$ stably divides with parameters, over $A$, if there is some tuple $d$ of parameters and a formula $\chi(x,z)$ with parameters from $d$ which is stable (does not have the order property) and such that $\psi_{i}(x)$ is equivalent to $\chi(x,c)$ for some $c$, and such that $\psi_{i}(x)$ divides over $(A,d)$. 

\begin{Proposition} \label{proposition:Theorem3.3} Let $A\subseteq B$. 
$tp(a/B)$ thorn forks over $A$ iff $tp(a/B)$ ``stably forks with parameters" over $A$.
\end{Proposition}

\begin{proof} First by the routine methods discussed earlier, it is enough to fix  some $(|T|+ |A|)^{+}$-saturated model $M$ containing $B$, let $p(x)$ be a complete type over $M$, and prove that 
\newline
(*) $p$ thorn divides over $A$ iff $p$ stably divides with parameters, over $A$. (Also in both cases we know by saturation that additional parameters witnessing strong dividing or stable dividing with parameters, can be found in $M$.)

The ``left implies right" direction of (*) follows immediately from the ``right implies left" direction of Proposition 3.1 (taking $B = M$).
This is because if $A\subseteq C \subseteq M$ and $X$ is a definable set in $p$ which is not almost over $C$ and is weakly normal relative to $C$, then first $X$ divides over $C$ by Remark 2.2, and secondly $X$ is an instance of an $L_{C}$-formula $\chi(x,y)$ which is stable, by Remark 2.4.  So $p$ stably forks with parameters, over $A$.

The ``right implies left" direction is proved in \cite{Ealy-Onshuus} using Theorem 5.1.1 in \cite{Onshuus}. Our upcoming proof of ``right implies left" will give another, fairly explicit, account of this Theorem 5.1.1 of \cite{Onshuus}. 
So let us assume the right hand side of (*). Let $d\in M$, let $\phi(x,y)$ be a formula over $d$ which is stable, let $\phi(x,c)\in p(x)$ for some $c\in M$ and suppose that $\phi(x,c)$ divides over $(A,d)$.  We want to find a formula in $p(x)$ which thorn divides over $A$.  We make use of local stability as in Chapter 1 of \cite{Pillay-book}.  Although the formula $\phi(x,y)$ is not over $\emptyset$, everything works as long as we make sure to work over the parameters $d$, which we do.  We consider $p|\phi$, the restriction of $p$ to a complete $\phi(x,y)$-type.  So $p|\phi$ divides over 
$(A,d)$. Let $e$ be the canonical base of $p|\phi$, so $e\notin acl(Ad)$.  By Lemma 1.2.(ii) in Chapter 1 of \cite{Pillay-book}, there are $a_{1},..,a_{n}\in M$ such that $a_{i+1}$ realizes the restriction of (the complete type) $p$ to $(A,d,e, a_{1},..,a_{i})$ and such that $e\in acl(A,d,a_{1},..,a_{n})$.  Choose $n$ minimal witnessed by $a_{1},..,a_{n}\in M$.  Let $d' = (a_{1},..,a_{n-1})$. 
Hence $e\notin acl(A,d,d')$, but $e\in acl(A,d,d',a)$ whenever $a$ realizes $p|(A,e,d,d')$.  By compactness there is a formula $\chi(x)\in p(x)|(A,e,d,d')$ such that $\chi(x)$ implies that  $e\in acl(A,d,d,x)$. Working back in the original language $L$, let $e'$ be the canonical parameter of $\chi(x)$ and without loss $\chi(x)$ is of the form $\chi(x,e')$ for some $L$-formula $\chi(x,w)$.  So $e'\in dcl(A,e,d,d')$.
\newline
{\em Claim 1.}  $e'\notin acl(A,d,d')$.
\newline
 Suppose otherwise that $e'\in acl(A,d,d')$. So, as $e\notin acl(A,d,d')$, it is easy to find some realization $a'$ of $\chi(x,e')$ with $e\notin acl(A,d,d',a')$, a contradiction.  So Claim 1 is established.

\vspace{2mm}
\noindent
{\em Claim 2.}  For any realization $a'$ of $\chi(x,e')$, $e'\in acl(A,d,d',a')$
\newline
{\em Proof of Claim 2.}  By choice of $\chi(x,e')$, for any realization $a'$ of $\chi(x,e')$, $e\in acl(A,d,d', a')$. But $e'\in dcl(A,e,d,d')$, hence, $e'\in acl(A,d,d', a')$.

\vspace{2mm}
\noindent
Claims 1 and 2 say that $\chi(x,e')$ strongly divides over $A,d,d'$.  So $\chi(x,e')$ thorn divides over $A$. As $\chi(x,e')\in p(x)$ this means that $p(x)$ thorn divides over $A$, completing the proof.

\end{proof}

\vspace{5mm}
\noindent
Finally we discuss Corollary 4.2 of \cite{Ealy-Onshuus}.  This states that if $\phi(x,b)$ thorn divides over $C$, then a tuple $d$ such that $\phi(x,b)$ strongly divides over $(C,d)$ can be found in $dcl(C,b)$. 

In fact the role of $C$ is also problematic.  Remember that $\phi(x,b)$ is defined to thorn divide over $C$ if there is $C'\supseteq C$ such that $\phi(x,b)$ strongly divides over $C'$  ($\phi(x,b)$ is not almost over $C'$ and the set of $C'$-conjugates of $\phi(x,b)$ is $k$-inconsistent for some $k$).  But then, if $\phi(x,b)$ thorn divides over $C$, it also thorn divides over $\emptyset$, and so Corollary 4.2 of \cite{Ealy-Onshuus} would yield that we can find $d\in dcl(b)$ such that $\phi(x,b)$ strongly divides over $d$.

Translating into the weakly normal language we obtain, via Remark 2.1:  
\begin{Question} Let $X$ be a definable set with canonical parameter $e$. Suppose that for some set of parameters $C$, $X$ is weakly normal, relative to $C$, and $e\notin acl(C)$. Can one find such a set of parameters $C\subseteq dcl(e)$?
\end{Question} 

We first give a couple of examples related to  Question 3.4 and its context. 
We start with an example giving  a positive answer for strongly minimal theories, to a slightly weaker question. 
This slightly weaker question, in the context of stable theories, states: suppose $p(x)$ is a stationary type which forks over $\emptyset$. Is there a definable set (formula) $X$ in $p$, with canonical parameter $e$, such that for some $C\subseteq dcl(e)$, $e\notin acl(C)$, and $X$ is weakly normal relative to $C$?

We will work with Morley rank which corresponds to algebraic independence dimension for tuples in  strongly minimal sets. 

\begin{Lemma} Suppose $T$ is a $1$-sorted strongly minimal theory with elimination of imaginaries. Let $p({\bar x})$ be a stationary complete type over a tuple ${\bar e}$ where ${\bar e}$ is the canonical base of $p$. Suppose that $p$ forks over $\emptyset$. Then there is a formula (definable set) $X$ in $p$ with canonical parameter $\bar{e}$ and some ${\bar c}\in dcl({\bar e})$ such that ${\bar e}\notin acl({\bar c})$ and $X$ is weakly normal relative to ${\bar c}$.
\end{Lemma}
\begin{proof} Thanks to the referee for the suggestions which gave rise to the following correct proof. 

Let ${\bar a}$ realize $p$. 
By our elimination of imaginaries assumption ${\bar e}$ can be taken to be a (finite) tuple from the strongly minimal home sort. Let ${\bar c}$ be a subtuple of ${\bar e}$ which is a basis of ${\bar e}$ over ${\bar a}$. Hence
\newline
{\em Claim I.}   ${\bar e}\in acl({\bar a}, {\bar c})$.

\vspace{2mm}
\noindent
Note that ${\bar c}$ is independent from $\bar a$ over $\bar \emptyset$, so as ${\bar a}$ is {\em not} independent from ${\bar e}$ over $\emptyset$ we have:
\newline
{\em Claim II.} ${\bar e}\notin acl({\bar c})$. 

Let ${\bar y}$ be a tuple of variables of length that of ${\bar e}$ and let ${\bar z}$ be the subtuple of ${\bar y}$ corresponding to the subtuple ${\bar c}$ of ${\bar e}$. 
By Claim I and compactness we can find a $L$-formula $\phi({\bar x},{\bar y})$ in $tp({\bar a}, {\bar e})$ such that
\newline
{\em Claim III.}  $\phi({\bar x}, {\bar y})$ implies ${\bar y}\in acl({\bar x}, {\bar z})$.
\newline
We may also assume that $RM(\phi({\bar x},{\bar e})) = RM(p({\bar x})) = n$, say,  and $dM(\phi({\bar x}, {\bar e})) = 1$ where $RM$ is Morley rank, and $dM$ is Morley degree. 

\vspace{2mm}
\noindent
{\em Claim IV.}  ${\bar e}$ is a canonical parameter of the set $X$ defined by $\phi({\bar x}, {\bar e})$.
\newline
{\em Proof of Claim IV.}  This follows from ${\bar e}$ being the canonical base of $p$ and from the assumptions on 
Morley rank and degree. Here are some details. Working in a suitably saturated model $M$ let  ${\bar e}'$ have the same type as ${\bar e}$ over 
$\emptyset$ and with ${\bar e}'\neq {\bar e}$. Suppose for a contradiction that $\phi({\bar x}, {\bar e})$ is {\em 
equivalent} to $\phi({\bar x}, {\bar e}')$. Let $p'({\bar x})\in S(M)$ be the unique nonforking extension of $p$ over $M$. 
Note that $p'$ is the unique complete type over $M$ with Morley rank $n$ and which contains the formula $\phi({\bar x}, 
{\bar e})$. Let $f$ be an automorphism of $M$ with $f({\bar e}) = {\bar e}'$. Then $f(p')$ is the unique complete type 
over $M$ of Morley rank $n$ and containing $\phi({\bar x}, {\bar e}')$. But as the latter formula is equivalent to $\phi({\bar x}, {\bar 
e})$, $f(p') = p'$, contradicting the characteristic property of ${\bar e}$ being the canonical base of $p'$.  We have proved that an automomorphism (of  an ambient saturated model) fixes ${\bar e}$ if and only if  it fixes $X$ setwise, as required.

\vspace{2mm}
\noindent
Now (III) above as well as (IV) implies that $X$ is weakly normal relative to ${\bar c}$. Together with Claim II, this completes the proof of the lemma, bearing in mind that $\bar c$ is a subtuple of $\bar e$.  
\end{proof} 

The next example is where both the hypothesis and conclusion of Question 3.3 fail, but the hypothesis``almost" holds.

We will take as $T$ the theory of the free pseudoplane.  See  Example 6.1 in Chapter 4 of \cite{Pillay-book} as well as Section 2 of \cite{Baudisch-Pillay} for some more details.  This is also called the infinite forest in \cite{Adler}.  The language consists of a single binary relation $I$, and the axioms for $T$ say that $I$ is symmetric, irreflexive, for each $a$ there are infinitely $b$ such that $I(a,b)$, and there are no ``loops", namely for each $n\geq 3$ there do not exist $a_{0}, a_{1},...., a_{n}$ such that $I(a_{i},a_{i+1})$ (for all $i=0,..,n-1$), the $a_{i}$ for $i\leq n-1$ are distinct, and $a_{0} = a_{n}$. 
$T$ is complete, and $\omega$-stable, where the Morley rank of the home sort is $\omega$.  A saturated model of $T$ consists of infinitely many connected components. In \cite{Baudisch-Pillay} it is pointed out that $T$ has weak elimination of imaginaries.

Fix a (saturated) model ${\bar M}$, let $a\in {\bar M}$, then the formula $I(x,a)$ isolates a complete type $p_{a}(x)$ over $a$.  $a$ is the canonical base of $p_{a}$ as well as the canonical parameter of the formula $I(x,a)$.

\begin{Lemma} (i) $I(x,a)$ is not weakly normal (i.e. relative to $\emptyset$).
\newline
(ii) There is no $c\in {\bar M^{eq}}$ such that $a\notin acl(c)$ and $I(x,a)$ is weakly normal relative to $c$.
\end{Lemma}
\begin{proof} (i) Fixing $b$ such that $I(a,b)$, $I(b,y)$ is infinite and isolates a complete type over $b$ (as mentioned above).  So $a\notin acl(b)$.
\newline
(ii)  Choose $c\in {\bar M}^{eq}$ such that $a\notin acl(c)$ and we want to show that $I(x,a)$ is not weakly normal relative to $c$.  By the weak elimination of imaginaries we may assume that $c$ is a real tuple.  The assumption that $a\notin acl(c)$ implies that there are no two elements $c_{1}, c_{2}$ from the tuple $c$ such that the unique shortest path from $c_{1}$ to $c_{2}$ goes via $a$. It follows that there is (unique) $b$ such that $I(a,b)$ and such that all elements of $c$ are on (shortest) paths from $a$ which go through $b$ (or in different components of the model). 
Now choosing this unique $b$ realizing $I(x,a)$, we see that $a\notin acl(b,c)$, because there are infinitely many $a'$ such that $I(b,a')$ and $a'$ is not on a path from $b$ to any element of $c$, and all such $a'$ have the same type over $b,c$. 
\end{proof}

\begin{Lemma}\label{lemma:pseudoplane2} Let $I(b,a)$. Let $\phi(x)$ be the formula $I(x,a)\wedge x\neq b$. Then for any $b'$ realizing $\phi(x)$ we have that $a\in acl(b,b')$.
\end{Lemma}
\begin{proof} $a$ is in the unique shortest path between $b$ and $b'$.
\end{proof}

However note that the canonical parameter of the formula $\phi(x)$ from Lemma \ref{lemma:pseudoplane2} is $(a,b)$ and the lemma says that $\phi(x)$ is weakly normal with respect to $b$.

\section{Theories with selectors} 
In this section we give a family of  negative answers to Question 3.4 (so counterexamples to 4.2 of \cite{Ealy-Onshuus}). 
(In \cite{Ealy-Onshuus} there is also a Theorem 4.1 from which their Corollary 4.2 is deduced without proof, and we assume that our examples also give negative answers to Theorem 4.1 of \cite{Ealy-Onshuus} although we did not, and do not want to, check details.)

It is a simple construction (maybe known) which for any theory $T$ produces a ``mild expansion" (in the sense of also adding a new sort) $T^{S}$, which for any infinite definable set $X$ in $T$ with canonical parameter $e$ yields some $d$ not in $dcl(e)$ such that $X$ is not almost over $d$ and $X$ is weakly normal relative to $d$ in $T^{S}$.  The construction is related to but distinct from  the {\em generic variations} of  \cite{Baudisch}.

We fix a complete theory $T$ in a language $L$, which we assume, for simplicity, to be relational.  There is no harm in assuming $T$ to be $1$-sorted.
We will define a language $L^{+}$ and complete $L^{+}$-theory $T^{+}$, and a language $L^{S}$ and complete $L^{S}$-theory $T^{S}$, all depending on $T$, and with $L^{+}\subseteq L^{S}$. 

Roughly speaking $T^{+}$ is the theory of a set equipped with an equivalence relation $E$ with infinitely many classes, such that each $E$-class  has structure making it a model of $T$, this is uniform across the $E$-classes, and moreover any interaction (on the level of the language $L$) between distinct $E$-classes is forbidden.
Model-theoretically $T^{+}$ is quite transparent, but one should give a definite formalism, which we do now. 

We discuss the language $L^{+}$ and theory $T^{+}$ simultaneously, sometimes mixing up syntax and semantics.  There will be two sorts in $L^{+}$, a sort $P$ equipped with an equivalence relation $E$, and the second sort is just $P/E$, and we have the canonical function $f_{E}:P \to P/E$, in the language $L^{+}$.

For $a$ of sort $P$, $a/E$ denotes the equivalence class of $a$ as an element of the sort $P/E$. And $[a]_{E}$ denotes the equivalence class of $a$ as a subset of the sort $P$. 

For each $n$-ary relation symbol $R$ of the language $L$, we will have an $(n+1)$-ary relation symbol 
$R^{+}(z,x_{1},..,x_{n})$ where $z$ is a variable of sort $P/E$ and $x_{1},..,x_{n}$ variables of sort $P$.  In addition to the equality symbols on the two sorts, this is the language $L^{+}$. The axioms of $T^{+}$ say that $f_{E}$ is what it should be, that for any $z\in P/E$, $R^{+}(z, x_{1},..,x_{n})$ implies that the $x_{i}$ are in the equivalence class determined by $z$, and that each $E$-class $C$ is a model of $T$. Where the $E$-class $[a]_{E}$ is viewed as an $L$-structure by defining $R(b_{1},..,b_{n})$ to hold in $C$ iff  $R^{+}(a/E,b_{1},..,b_{n})$ holds in the ambient $L^{+}$-structure.

To summarise, $L^{+}$ is the language with sorts $P$ and $P/E$, and symbols $E$, $f_{E}$ and the  $R^{+}$ for $R\in L$.  
\begin{Lemma} (i) $T^{+}$ is complete,
\newline
(ii)  Assuming that $T$ has quantifier elimination (i.e. is Morleyized), then $T^{+}$ has quantifier elimination too.
\newline
(iii) $P/E$ is an indiscernible set in $T^{+}$.
\end{Lemma}
\begin{proof} (i) is clear as a model of $T^{+}$ is just a family of models of $T$ indexed by $P/E$, and any two saturated (in the strict sense) models will be isomorphic. For if $M$ is $\kappa$-saturated of cardinality $\kappa$ (where $\kappa > |L^{+}|$) then there will be $\kappa$ $E$-classes and moreover each $E$-class will be a $\kappa$-saturated model of $T$ of cardinality $\kappa$.  Alternatively (without using the existence of saturated models in the strict sense) one can just do a back-and-forth between any two $\omega$-saturated models of $T^{+}$
\newline
(ii) follows by either observing that if $M$ is $\kappa$-saturated of cardinality $\kappa$ then any permutation of the $P/E$ sort extends to an automorphism of $M$, or doing a back-and-forth argument between $\omega$-saturated models (or in a fixed $\omega$-saturated model).

\end{proof}

Of course we can assume that $T$ has QE, so we obtain complete understanding of definability in models of $T^{+}$. 
Note that if $M$ is a saturated model of $T$  (in the weak or strong sense)  then  the model $M^{+}$ of $T^{+}$ where $E$ has $\kappa$-many classes and each class is isomorphic (as an $L$-structure) to $M$, is also saturated, in the appropriate sense.   We write $P(M^{+})$ and $(P/E)(M^{+})$ for the corresponding sorts in  $M^{+}$.

One immediate observation which will be useful for us later is:

\begin{Lemma} Let $X\subset P(M^{+})$ be definable in $M^{+}$ over $a$ from $P(M^{+})$. Let $C$ be the $E$-class (as a subset of $P(M^{+})$) of $a$. Then $X = X_{1}\cup X_{2}$ where $X_{1}$ is a subset of $C$ definable over $a$ in $L$, and $X_{2}$ is the union of a  family of uniformly $\emptyset$-definable (in $L$) subsets (maybe empty) of the classes other than $C$. Moreover if $a$ is the canonical parameter of $X_{1}$ in $C$ (as a model of $T$), then $a$ is the canonical parameter of $X$ in $T^{+}$. 
\end{Lemma}

We now introduce the language $L^{S}$ and theory $T^{S}$.  $L^{S}$ is $L^{+}$ together with a new sort $Q$ and function symbol $\pi: P/E \times Q \to P$.
$T^{S}$ can be described in two equivalent ways. First take any model $M$ of $T$ and expand $M^{+}$ to an $L^{S}$ structure $M^{S}$ by letting the sort $Q(M^{S})$ consists of all ``selectors"  from the equivalence classes, and $\pi(M^{S})$ the obvious thing.  Namely, $Q(M^{S})$ is the collection of all choice functions or sections $s$ corresponding to the map $f_{E}:P\to P/E$ and $\pi(z,s) = s(z)$. Put $T^{S} = Th(M^{S})$. 

Alternatively $T^{S}$ is the $L^{S}$-theory expanding $T^{+}$ which says: for $y_{1}, y_{2}\in Q$, $y_{1} = y_{2}$ iff  $\pi(z,y_{1}) = \pi(z,y_{2})$ for all $z$, together with the axiom $f_{E}(\pi(z,y)) = z$, as well as the axiom which says that for all $d\in Q$, $z\in P/E$ and $a\in P$ in the class named by $z$, there is $d'\in Q$ such that $\pi(z,d') = a$ and $\pi(z',d') = \pi(z',d)$ for all $z' \neq z$.

Iterating the last axiom, it implies that for any $d$ in $Q$ and finitely many $E$-classes we can find $d'\in Q$ such that the value of $\pi(-,d')$ on these finitely many classes is anything one wants, but for the other classes is the same as the value of $\pi(-,d)$.  (And in a $\kappa$-saturated model we can do it for $<\kappa$ many equivalence classes instead of only finitely many.) 

\begin{Remark} The theory $T^{S}$ as defined by the axioms and observation in the last two paragraphs is complete with quantifier elimination, from which it follows that $T^{S}$ is indeed equal to $Th(M^{S})$ as defined earlier where $M$ is any model of $T$.
\end{Remark}
\begin{proof} This follows by a back-and-forth argument between any two $\omega$-saturated models of the axioms.
\end{proof}

\begin{Lemma}  Let $N$ be a model of $T^{+}$ and ${\bar a}$, ${\bar b}$ tuples of the same length (and in appropriate sorts) from $N$ such that $tp_{N}({\bar a}) = tp_{N}({\bar b})$.  Let $N'$ be an expansion of $N$ to a model of $T^{S}$. Then $tp_{N'}({\bar a}) = tp_{N'}({\bar b})$.
\end{Lemma}

\begin{proof} We will prove the special case where ${\bar a}$ (and ${\bar b}$) are singletons from $P$.  The general case is similar.
We may assume that everybody is saturated.
First suppose that $a$ and $b$ are in the same $E$-class $C$ say. So viewing $C$ as a model of $T$, there is an automorphism $f$ of $C$ taking $a$ to $b$.  Then $f$ extends to an automorphism of $N$ by fixing $P/E$ pointwise and fixing pointwise every $E$-class $C'\subseteq P$ other than $C$.
We will call this automorphism $f$ too. We now want to extend $f$ to an automorphism $f'$ of $N'$ by defining the action on $Q$.   For each $d\in Q$, let $f'(d)$ be the unique element $d'\in Q$ such that $\pi(a/E,d') = f(\pi(a/E,d))$ and $\pi(z,d') = \pi(z,d)$ for all $z\neq a/E$ (where $d'$ is given to us by the axioms above).  It is then easy to see that $f'$ is a bijection of $Q$ with itself and is an automorphism of  $N'$. So $a$ and $b$ have the same type in $N'$. 

Now suppose that $a$ and $b$ are in different $E$-classes. The quantifier elimination result earlier for $T^{+}$ implies that not only do $a$ and $b$ have the same types in $N$ but $(a,b)$ and $(b,a)$ have the same types in $N$. So there is 
an automorphism $f$  of $N$ taking $(a,b)$ to $(b,a)$  (so also switching $a/E$ and $b/E$) and fixing pointwise all other elements of $P/E$ and all $E$-classes $C$ other than $[a]_{E}$ and $[b]_{E}$.  As in the  first paragraph, $f$ extends to an automorphism $f'$ of $N'$. So in this case $(a,b)$ and $(b,a)$ have the same type in $N'$.
\end{proof}

It follows that:
\begin{Corollary} The (relativised) reduct $T^{+}$ of $T^{S}$ is ``weakly stable embedded" in the sense that for subsets of Cartesian products of the $P/E$ and $P$, \\$\emptyset$-definability in $T^{+}$ coincides with $\emptyset$-definability in $T^{S}$.   In particular the sort $P/E$ is an indiscernible set in the theory $T^{S}$.
\end{Corollary} 

On the other hand $T^{+}$ is not ``stably embedded" in $T^{S}$, as  there will be sets in the $T^{+}$-sorts which are definable with parameters in the $Q$ sorts but not in the $T^{+}$ sorts.  In fact:

\begin{Proposition} $T^{S}$ has both the strict order property and the independence property, all witnessed on the $P/E$ sort (which remember is an indiscernible set in $T^{+}$).
\end{Proposition} 
\begin{proof}  Let $N$ be any model of $T^{S}$  (saturated if one wishes), and fix $d\in Q(N)$ which gives a section $\pi(-,d)$ from $(P/E)(N)$ to $P(N)$.  For any finite subset $S$ of $P/E$, we can (by the axioms for example) find an element $d_{S}\in Q(N)$ such that  $\pi(z,d) = \pi(z,d_{S})$ precisely for those $z\notin S$.  Namely finite subsets of 
$P/E$ are uniformly definable, giving both the strict order property and the independence property.
\end{proof} 

We now point out a general result which will be useful in our (counter) examples. Note that by 4.4 there is a unique $1$-type of an element of $P/E$ in $T^{S}$. 

\begin{Proposition} Suppose that there is a unique $1$-type (over $\emptyset$) in $T$.  Then (working in a saturated model of $T^{S}$), for every $a\in P$, there is $d\in Q$ such that $\pi(a/E, d) = a$, $a\notin acl(d)$ and $d\notin acl(a)$.
\end{Proposition}
\begin{proof} It is enough to work in a model $M^{S}$ as described above where $M$ is a $\kappa$-saturated, strongly $\kappa$-homogeneous model of $T$ ($\kappa > |T|$), $M^{+}$ is the model of $T^{+}$ consisting of $\kappa$ many ``copies" of $M$, and $M^{S}$ is the expansion of $M^{+}$ to a model of $T^{S}$ by adding all selectors (or sections) as $Q(M^{S})$. Note that using quantifier-elimination $M^{S}$ is also $\kappa$-saturated.

First, as there are infinitely many $d\in Q(M^{S})$ such that $\pi(a/E, d) = a$ we may choose such $d\notin acl(a)$. 
Let $e = a/E$. Let $e'\neq e$ in $(P/E)(M^{S})$. And let  $b = \pi(e',d)$.  Viewing $[a]_{E}$ and $[b]_{E}$ as models of $T
$ there is an isomorphism $f$ between them taking $a$ to $b$ (by our assumptions).  Then $f\cup f^{-1}$ extends to an automorphism of $M^{+}$ fixing all $E$-classes pointwise, other than $[a]_{E}$ and $[b]_{E}$.  This induces a bijection from $Q(M^{S})$ to itself which fixes $d$, hence an automorphism $f'$ of $M^{S}$. Hence $a$ and $b$ have the same type over $d$ in $M^{S}$. As $e'\neq e$ was arbitrary this shows that $a\notin acl(d)$ in $M^{S}$. 
\end{proof} 

The results of this section lead to many examples yielding negative answers to Question 3.4. We will mention two of them. 
Let us begin with $T = DLO$ the theory of dense linear orderings without endpoints in the language $\{<\}$. The theory $T$ is complete with quantifier elimination and a unique $1$-type over $\emptyset$. $T$ also has elimination of imaginaries. Let $M$ be a (saturated) model of $T$ and $a\in M$ and consider the formula $x>a$ and the set $X$ it defines in $M$. Then $X$ is not weakly normal. Moreover, $a$ is a canonical parameter for $X$ and the only element in $dcl(a)$ is $a$, hence $X$ is not weakly normal relative to any set $C$ such that $X$ is not almost over $C$. 

Now let $M^{S}$ be the model of $T^{S}$ built from $M$ as at the beginning of the proof of Proposition 4.7. Fix an $E$-class $C\subset P(M^{S})$ which we identify with $M$ and let $X$ be as above, now considered as a definable set in $M^{S}$. 

\begin{Proposition}  Working in $M^{S}\models T^{S}$ and $X$ as above:
\newline
The canonical parameter of $X$ is $a$. For no $e\in dcl^{eq}(a)$ do we have that $X$ is weakly normal relative to $e$ and $a\notin  acl(e)$. But there is $d\in M^{S}$ such that $X$ is weakly normal relative to $d$ and $a\notin acl(d)$.
\end{Proposition}
\begin{proof} First, by Proposition 4.7 we find $d\in Q$ such that $\pi(a/E,d) = a$ and $a\notin acl(d)$ and $d\notin acl(a)$. In particular for any $b\in X$, $a\in dcl(b,d)$. Hence $X$ is weakly normal relative to $d$, and $a\notin acl(d)$. (All in the sense of $T^{S}$.)  
Now we want to check that $a$ is a canonical parameter of $X$ in the sense of $T^{S}$.  This follows immediately from the fact that $T^{S}$ is an expansion of $T$ and that $a$ is a canonical parameter of $X$ in $T$.

Now, we want to check that: for no $e\in dcl^{eq}(a)$ do we have that $X$ is weakly normal relative to $e$ and $a\notin  acl(e)$.  So suppose $e\in dcl^{eq}(a)$ in $M^{S}$, so $h(a) = e$ for some $\emptyset$-definable function $h$ in $(M^{S})^{eq}$. But then $h(x) = h(y)$ is an $L^{S}$-formula $\phi(x,y)$ on the sort $P$. And clearly the canonical parameter of the formula $\phi(x,a)$ is interdefinable with $e$.  We will now use Corollary 4.5 and Lemma 4.2. First by Corollary 4.5, $\phi(x,a)$ is defined in $M^{+}$ over $a$ by a formula which we still call $\phi(x,a)$, and moreover a canonical parameter of $\phi(x,a)$ is a canonical parameter for it in $T^{S}$.  Let $C$ be the $E$-class of $a$.
By 4.2 the canonical parameter of $\phi(x,a)$ in $T^{+}$ is the same as the canonical parameter of $\phi(x,a)\cap C$ in the model $C$ of $T = DLO$ (where $\phi(x,a)\cap C$ is definable in $C\models DLO$ by a quantifier-free formula with parameter $a$). Clearly the canonical parameter in $DLO$ of this definable set is $a$ or $\emptyset$. 
So we have shown that  $e$ is interdefinable with either $a$ or $\emptyset$ in $T^{S}$. 
In the first case of course $a\in acl(e)$.  In the second case, as $X$ is not weakly normal in  $T$ it is not weakly normal in $T^{+}$, hence (by 4.4 or 4.5) not weakly normal in $T^{S}$  (i.e. all relative to $\emptyset$).

\end{proof}

Bearing in mind our discussion of the theory $T$ of the free pseudoplane at the end of Section 3,  and choosing now $X$ to be defined by $I(x,a)$ in a saturated model $M$ of $T$, an identical proof to that of Proposition 4.8 above yields a definable set $X$ in $M^{S}$ giving also a negative answer to Question 3.4 (and a counterexample to 4.2 of \cite{Ealy-Onshuus}).


\begin{thebibliography}{99}

\bibitem{Adler} H. Adler, A geometric introduction to forking and thorn forking, Journal Math. Logic, vol. 9 (2006), 1 - 20.

\bibitem{Baudisch}  A. Baudisch, Generic variations of models of T,  JSL,  67 (2002), 1025–1038. 

\bibitem{Baudisch-Pillay}  A. Baudisch and A. Pillay, A free pseudospace,  JSL, 65 (2000), 443 - 460. 

\bibitem{Casanovas} E. Casanovas,  Simple theories and hyperimaginaries. Lecture Notes in Logic, 39,  ASL-CUP, 2011.

\bibitem{Ealy-Onshuus}  C. Ealy and A. Onshuus, Thorn forking and stable forking, Rev. Acad Colomb. Ciene. Ex. Fis Nat. 40(157), 2016, 683 - 689. 

\bibitem{Halevi-Kaplan} Y. Halevi and I. Kaplan, Saturated models for the working model theorist, arXiv:211202775v1.

\bibitem{Onshuus} A. Onshuus, Properies and consequences of thorn forking, JSL, 71 (2006), 1 - 21.

\bibitem{Pillay-book} A. Pillay, Geometric Stability Theory, Oxford University Press, 1996. 

\bibitem{Srour}  G. Srour, The notion of independence in categories of algebraic structures, part I: Basic properties, Annals of Pure and Applied Logic, 38, Issue 2, May 1988, Pages 185-213.






\end{thebibliography}
\end{document}